\documentclass[12pt]{amsart} 
\usepackage{amsmath,amssymb,amsthm,graphicx}

\newtheorem{thm}{Theorem}[section]
\newtheorem{prop}[thm]{Proposition}
\newtheorem{lem}[thm]{Lemma}

\theoremstyle{definition}

\newtheorem{defn}[thm]{Definition}
\newtheorem{notation}[thm]{Notation}
\newtheorem{example}[thm]{Example}

\theoremstyle{remark}

\newtheorem{rem}[thm]{Remark}

\numberwithin{equation}{section}

\newcommand{\re}{\mathbb{R}}

\newcommand{\qu}{\mathbb{Q}}
\newcommand{\co}{\mathbb{C}}

\newcommand{\vv}{\mathcal{Z}}

\newcommand{\Bb}{\mathcal{B}}

\newcommand{\glb}{\mbox{\rm glb}}
\newcommand{\z}{\bar z}

\newcommand{\rp}{\mbox{Re}}
\newcommand{\ip}{\mbox{Im}}

\title[Constructing Bounded Rational Functions]{Constructing Discontinuous but Locally Bounded Rational Functions using \L ojasiewicz Inequalities}

\author{Adam Coffman}\thanks{ORCID 0000-0002-1437-7525}
\email{CoffmanA@pfw.edu \ \ \ \ \ \ \ Pan1@pfw.edu}
\author{Yifei Pan}

\address{Department of Mathematical Sciences \\ Purdue University Fort Wayne \\ 2101 E.\ Coliseum Blvd.\ \\ Fort Wayne, IN 46805 USA}


\begin{document}

\begin{abstract}
    For real multivariate polynomials $P$ and $Q$ both vanishing at a point, if the zero set of $Q$ is contained in the zero set of $P$, then there exists a rational function of the form $P^{p}/Q^{q}$ which is locally bounded and such that its extension that vanishes on the zero set of $Q$ is discontinuous.  The proof uses inequalities of \L ojasiewicz.
\end{abstract}

\maketitle

%

\section{Introduction and examples}\label{ir}

Real valued rational functions are familiar as counterexamples in multivariable calculus.  Functions with interesting or unexpected behavior are easily illustrated in the textbook and classroom by functions with two real inputs and one real output, and good choices of polynomial numerator and denominator.  We start with a few examples, with the goal of finding constructions for other bounded rational functions with discontinuities.
\begin{example}\label{ex1.1}
  The classic example $$f(x,y)=\frac{xy}{x^2+y^2}$$ is bounded but discontinuous, even when extended to have value $f(0,0)=0$, with different limit values along lines through the origin.  However its partial derivatives $\partial f/\partial x$ and $\partial f/\partial y$ exist at every point in $\re^2$.  Its graph in $\re^3$ is contained in the ruled surface defined by $z(x^2+y^2)-xy=0$ (known as ``Pl\"ucker's conoid,'' and of Type $8$ in the \cite{css} classification of real surfaces with quadratic rational parametrization).  See Figure 1 (left). 
  \begin{figure*}
    \includegraphics[scale=0.25]{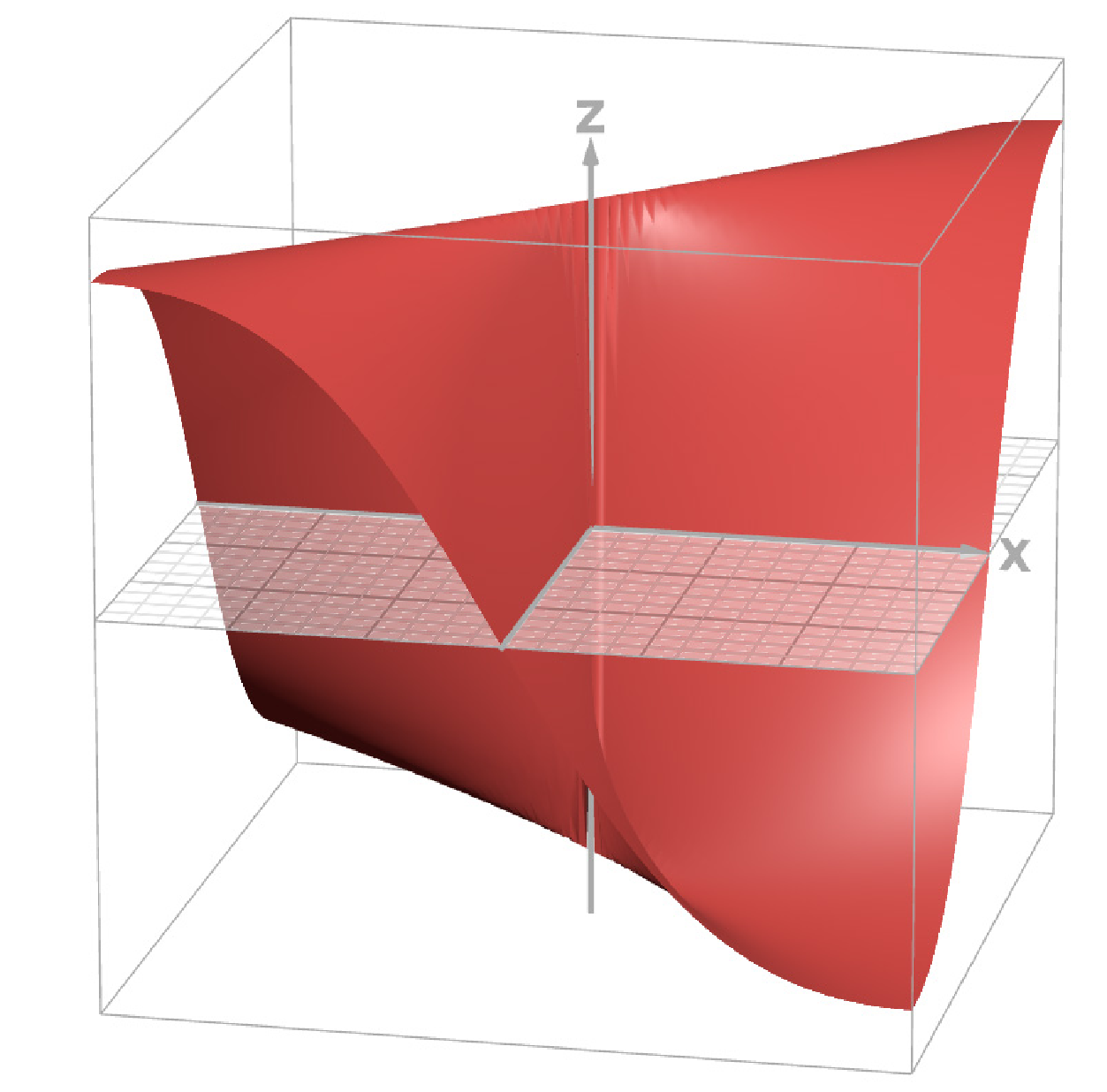} \ \ \ \ \ \ \  \includegraphics[scale=0.25]{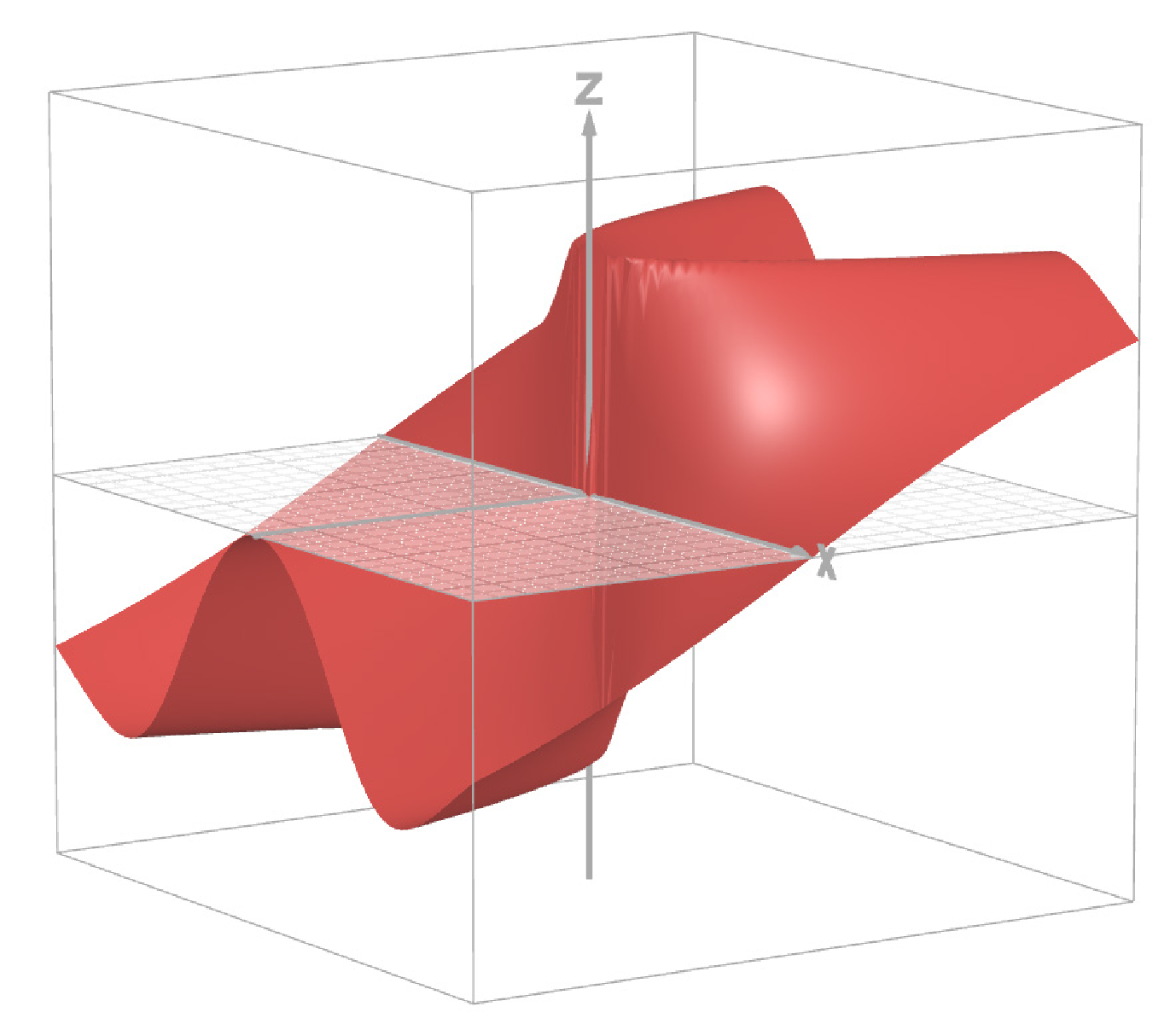}
    \caption{{$\displaystyle{z=\frac{xy}{x^2+y^2}}$.}  \ \ \ \ \ \ \ \ \ \ \ \ \ \ $\displaystyle{z=\frac{2x^2y}{x^4+y^2}}$.}\ \ \ \ \ \ \ \ \ \ \ \ \ \ \ \ \ 
  \end{figure*}
\end{example}
\begin{example}\label{ex1.2}
    The discontinuous function 
    \begin{equation}\label{eq1.2}
      f(x,y)=\frac{2x^2y}{x^4+y^2},
      \end{equation}
      again extended to $f(0,0)=0$, has limit $0$ as $(x,y)$ approaches $(0,0)$ along any straight line, but has limit $1$ along the curve $y=x^2$.  See Figure 1 (right).  The expression  (\ref{eq1.2}) is attributed to Genocchi and Peano in the survey paper \cite{cm}; for more examples with similar properties, see \cite{cm2} and \cite{cggr} Chapter 34. 
\end{example}
\begin{example}\label{ex1.3}
  The rational function (\cite{ib})
  \begin{equation}\label{eq1.3}
    f(x,y)=\frac{(x^2+y^2)^7}{(y^3-x^5)^2+(y-x^2)^8}
  \end{equation}
  is continuous  on a neighborhood of $(0,0)$ when extended to value $0$ at the origin, but not locally {\underline{Lipschitz continuous}} at $(0,0)$; there is no constant $C$ so that for $(x,y)$ near $(0,0)$, $|f(x,y)|\le C|(x,y)|$, but instead $f$ satisfies $f\approx x^{2/3}$ on the curve $y^3=x^5$.
\end{example}
In all these examples, $f$ is bounded near the origin even though the denominator approaches $0$.  In this article, we consider how local boundedness and continuity are related for rational functions of real variables.  Given polynomials $P$ and $Q$ with $P(\vec0)=Q(\vec0)=0$, we suppose that the zero set of $Q$ is contained in the zero set of $P$.  Then, for rational numbers $q/p\ge0$, the general idea is that $P/|Q|^{q/p}$, or equivalently $P^{2p}/Q^{2q}$, is unbounded for large $q/p$, but for small $q/p$, the rational function $P^{2p}/Q^{2q}$ has limit $0$ at $\vec0$ and extends to a continuous function near the origin.  Example \ref{ex1.1} and Example \ref{ex1.2} show there is an intermediate case, where the expression is locally bounded but its extension that vanishes on the zero set of $Q$ is not a continuous function.  Our main result is that this always happens: there is a rational exponent $p/q$ so that near $\vec0$, $P/|Q|^{q/p}$ is bounded but is not continuous in a neighborhood of the origin.  Our approach is elementary but at a few points we refer to \cite{bcr} for some foundational facts from the subject of real algebraic geometry, and the main results (Lemma \ref{lem3.3}, Theorem \ref{lem2.9}, and Theorem \ref{thm3.3}) rely on well-known inequalities of \L ojasiewicz.  The Proof of Theorem \ref{lem2.9} in fact shows that the existence of exponents so that $P^{2p}/Q^{2e}$ is bounded and discontinuous is equivalent to a strong version of the \L ojasiewicz inequality for pairs of polynomials $(P,Q)$.

In Section \ref{sec3}, the case where $Q$ is a power of the norm of the gradient $\vec\nabla P$ gives a simple way to construct a large class of locally bounded rational functions.  Although $Q=|\vec\nabla P|^2$ has isolated zeros for a generic polynomial $P$, Theorem \ref{ex4.11} shows how bounded rational functions can have an indeterminacy ($\frac00$) set with large dimension.

In Section \ref{sec4} we demonstrate another construction of bounded real rational functions, as the real or imaginary part of a complex valued expression.  This is another source of many examples and does not use the real algebraic geometry methods of Section \ref{sec2} or Section \ref{sec3}.

We remark that the related question of computing the value of a limit of a multivariable rational function is of continuing interest in the field of computer algebra, see \cite{a}, \cite{dgnp}, and their references.


\section{Elementary properties of rational functions}\label{rf}

We will work with real valued functions depending on $n$ real variables  $\vec x=(x_1,\ldots,x_n)$.  To fix notation we start by recalling some elementary properties of real polynomial and rational functions.

\begin{notation}\label{not1.2}
  The {\underline{zero set}} of a polynomial $P=P(x_1,\ldots,x_n)$ is denoted $$\vv(P)=\{\vec x:P(\vec x)=0\}\subseteq\re^n.$$ 
  We will also call a set of this form a {\underline{real algebraic set}}, and recall that the intersection of finitely many zero sets $\vv(P_1)\cap\ldots\cap\vv(P_N)$ is itself the zero set of one polynomial $P_1^2+\cdots+P_N^2$. 
\end{notation}
\begin{defn}\label{def1.2}
A {\underline{nonsingular}} point of a polynomial $P$ is $\vec x$ such that the gradient vector at $\vec x$ is non-zero: $$\vec\nabla P(\vec x)=\left.\left(\frac{\partial P}{\partial x_1},\ldots,\frac{\partial P}{\partial x_n}\right)\right|_{\vec x}\ne\vec0.$$
A point $\vec x_0$ such that $\vec\nabla P(\vec x_0)=\vec0$ is a {\underline{critical point}} of $P$.
\end{defn}  
At a point which is a nonsingular point of $P$  and also an element of $\vv(P)$, there is a neighborhood $U$ of $\vec x$ so that $U\cap\vv(P)$ is a smooth submanifold of dimension $n-1$ in $\re^n$. 
\begin{defn}\label{def1.3}
  A {\underline{rational function}} is defined as an expression of the form $$f(x_1,\ldots x_n)=\frac{P(x_1,\ldots,x_n)}{Q(x_1,\ldots,x_n)},$$
  where $P$ and $Q$ are polynomials and $Q\not\equiv0$.
  The {\underline{domain}} of a rational function $f=P/Q$ is the complement of $\vv(Q)$ in $\re^n$, and the {\underline{indeterminacy set}} of $f=P/Q$  is defined by $\vv(P)\cap\vv(Q)$.
\end{defn}
\begin{defn}\label{def2.4}
  For a subset $\Delta\subseteq\re^n$, a rational function $f=P/Q$ is {\underline{bounded}} on $\Delta$ means that there is some constant $M$ so that for all $\vec x$ in $\Delta\setminus\vv(Q)$, $$\left|\frac{P(\vec x)}{Q(x)}\right|\le M.$$  The function is {\underline{locally bounded}} near a point $\vec x$ means that it is bounded on some open neighborhood of $\vec x$.
\end{defn}
A rational function will not be locally bounded near a point where the denominator is zero, unless the numerator is also zero, as shown by the following Lemma.
\begin{lem}\label{lem2.5}
  If a rational function $P/Q$ is bounded on some open set $\Delta$, then $\Delta\cap\vv(Q)\subseteq\Delta\cap\vv(P)$.
\end{lem}
\begin{proof}
      For any point $\vec x\in\vv(Q)$, $\vec x$ is not an interior point of $\vv(Q)$ (\cite{bcr} Proposition 2.8.4).  If $\Delta\cap\vv(Q)$ is empty, then the claim is trivial; otherwise, for $\vec x\in\Delta\cap\vv(Q)$ there is some sequence of points $\vec x_i\in\Delta\setminus\vv(Q)$ converging to $\vec x$. This limit exists: 
  \begin{equation}\label{eq1}
    \lim_{i\to\infty} P(\vec x_i)=\lim_{i\to\infty}\left(\frac{P(\vec x_i)}{Q(\vec x_i)}\right)Q(\vec x_i)=0,
    \end{equation}
    since it is the product of a bounded sequence and a sequence with limit $0$.  By continuity of $P$, $P(\vec x)=0$, which shows $\Delta\cap\vv(Q)\subseteq\Delta\cap\vv(P)$.
\end{proof}

Example \ref{ex1.1} and Example \ref{ex1.2} showed that locally bounded rational functions may be discontinuous at points where the denominator is zero, but the following Lemma shows that multiplying by any small power of the denominator gives a continuous (but not necessarily rational) function.


\begin{lem}\label{lem3.4}
  Let $\Delta$ be any open set in $\re^n$.  If the rational function $P(\vec x)/Q(\vec x)$ is bounded on $\Delta$, then for any $\tau>0$, the function   $$S(\vec x)=\left\{\begin{array}{ll}{\displaystyle{\frac{P(\vec x)}{|Q(\vec x)|^{1-\tau}}}} & \mbox{\rm{for} $\vec x\in\Delta\setminus\vv(Q)$}\\ & \\0& \mbox{\rm{for} $\vec x\in\Delta\cap\vv(Q)$}\end{array}\right.$$
 is continuous on $\Delta$.
\end{lem}
\begin{proof}
    $P/|Q|^{1-\tau}$ is continuous by elementary calculus at points in the open set $\Delta\setminus\vv(Q)$.  $|Q|^\tau$ is continuous on $\re^n$, and if $Q$ is always nonnegative, or $\tau=p/q$ is rational with $q$ odd, then we do not need the absolute value in the statement or proof.  
    
    Let $M>0$ be a bound so that $|P(\vec x)/Q(\vec x)|\le M$ for $\vec x\in\Delta\setminus\vv(Q)$, and let $\epsilon>0$.  We denote the usual norm of a vector by $|(x_1,\ldots,x_n)|=\sqrt{x_1^2+\cdots x_n^2}$.  If $\vec x_0\in\Delta\cap\vv(Q)$, then there is some $\delta>0$ so that if $|\vec x-\vec x_0|<\delta$, then $\vec x\in\Delta$ and $|Q(\vec x)|^\tau<\epsilon/M$.  For $\vec x\in\vv(Q)$, $S(\vec x)=S(\vec x_0)=0$.  If $\vec x\not\in\vv(Q)$ and $|\vec x-\vec x_0|<\delta$, then $|S(\vec x)-S(\vec x_0)|=(|P(\vec x)|/|Q(\vec x)|)\cdot|Q(\vec x)|^\tau<M\cdot(\epsilon/M)=\epsilon$.

    We note that this proof only used the continuity of $P$ and $Q$, not the assumption that they are polynomials.
\end{proof}

The following Proposition states, as mentioned in the Introduction, that for any non-zero rational function $P/Q$, replacing $Q$ with some large power gives an unbounded rational function $P/Q^{\xi}$.
\begin{prop}\label{lem2.7}
  For an open set $\Delta\subseteq\re^n$, if there is a point $\vec x_0\in\Delta$ and a polynomial $Q\not\equiv0$ with $Q(\vec x_0)=0$, then for any polynomial $P\not\equiv0$ there exists a number $\xi_0>0$ so that the function $\displaystyle{\frac{P(\vec x)}{|Q(\vec x)|^\xi}}$ is not bounded on $\Delta\setminus\vv(Q)$ for any $\xi\ge\xi_0$.
%
\end{prop}
\begin{proof}   
  We refer to the Addendum \cite{cp} for the brief and elementary proof.
%
%
\end{proof}

\section{Recalling some real algebraic geometry}

\begin{notation}\label{not4.1}
  For $r>0$ and $\vec x_0\in\re^n$, let $B_r(\vec x_0)$ denote the Euclidean open ball in $\re^n$ with center $\vec x_0$ and radius $r$.  This will often be abbreviated $B_r$.  The closed ball is denoted $\overline B_r(\vec x_0)=\overline B_r$.
\end{notation}

Definition \ref{def1.3} allows the numerator and denominator, $P$ and $Q$, of a rational function to have a common polynomial factor $F$, and the zero set of such a common factor is part of the indeterminacy set of $P/Q$.  In fact, this is related to the earlier observations about dimension and the gradient.


\begin{prop}\label{prop2.4}
  For a rational function $P/Q$ with $P\not\equiv0$, the following are equivalent.
   \begin{enumerate}
      \item There exist polynomials $M_1$, $M_2$, and $F$, so that $F$ is irreducible, $\vv(F)$ has a non-singular point of $F$, and so that $P=M_1F$, and $Q=M_2F$.
      \item There is some point $\vec x$ in the indeterminacy set of $P/Q$ where the intersection of the set with a ball $B_r(\vec x)$ is a smooth submanifold of the ball, with dimension $n-1$.
  \end{enumerate}
  If, in addition, for every point $\vec x\in\vv(Q)$, $P/Q$ is bounded on some neighborhood of $\vec x$, then $(1)$ and $(2)$ are equivalent to:
  \begin{enumerate}\setcounter{enumi}{2}
      \item There is some point $\vec x\in\vv(Q)$ where the intersection $\vv(Q)\cap B_r(\vec x)$ is a smooth submanifold of the ball, with dimension $n-1$.
  \end{enumerate}
\end{prop}
\begin{proof}[Sketch]
  For the claim $(1)\iff(2)$, we refer to \cite{dfp} and to supplemental notes in our Addendum \cite{cp}.  We refer to the first few chapters of \cite{bcr} for a general detailed development of algebraic and topological notions of dimension in the subject of real algebraic geometry.  The rough idea of the Proposition is that $\vv(P)$ and $\vv(Q)$ generally have a lower-dimensional intersection unless $P$ and $Q$ have a common factor.

  To show $(2)\iff(3)$ under the boundedness assumption, for any point $\vec x\in\vv(Q)$, there is some neighborhood $\Delta$ where $\Delta\cap\vv(Q)\subseteq\Delta\cap\vv(P)$ by Lemma \ref{lem2.5}.  It follows that $\vv(Q)\subseteq\vv(P)$ and the indeterminacy set coincides with the zero set of the denominator: $\vv(P)\cap\vv(Q)=\vv(Q)$.
\end{proof}
In Proposition \ref{prop2.4}, the common factor $F$ can be canceled from the fraction $P/Q$ to get $M_1/M_2$, although we consider these expressions to be different as rational functions.  A contrapositive restatement of the Proposition is that a rational function $\not\equiv0$ with no common factors in the numerator and denominator must have an indeterminacy set with dimension $n-2$ or smaller.  This was observed in the examples from Section \ref{ir}, where $n=2$ and the denominators had only isolated zeros.

\begin{defn}\label{def4.2}
  A subset of $\re^{n}$ is a {\underline{semi-algebraic}} set means that it is a finite union of finite intersections of sets of the form $\{F=0\}$ or $\{F>0\}$ for polynomials $F(\vec x)$.  For a subset $\Delta\subseteq\re^{n_1}$, a function $f:\Delta\to\re^{n_2}$ is a {\underline{semi-algebraic}} function means that its graph is a semi-algebraic set in $\re^{n_1+n_2}$.
\end{defn}

The next Proposition states a well-known local description of half-branches of real algebraic plane curves, formulated specifically for application in Theorem \ref{lem2.9}.


\begin{prop}\label{lem2.11}
  Given a semi-algebraic set $X\subseteq\re^2$ with dimension $1$, if $X$ contains $(0,0)$ then there exist $\delta>0$, $\eta>0$, $K\ge0$ so that the intersection of $X$ with the open box $(0,\delta)\times(0,\eta)$ is a disjoint, finite union of graphs of increasing functions, $b_1,\ldots,b_K$ so that for each $k$, ${\displaystyle{\lim_{u\to0^+}b_k(u)=0}}$ and $b_k:(0,\delta)\to(0,\eta)$ is of the form $\tilde b_k(u^{1/e_k})$ where $e_k$ is a positive integer and $\tilde b_k$ is real analytic on $(-\delta^{1/e_k},\delta^{1/e_k})$.  
\end{prop}
\begin{proof}
  We refer to the Addendum \cite{cp} for a proof.
\end{proof}

The main steps in the proofs in the next Sections use the circle of ideas generally known as \L ojasiewicz inequalities, which we break up into some specific claims.

\begin{prop}[\L ojasiewicz Inequality]\label{prop6.1}
  Given polynomials $P$ and $Q$ and a point $\vec x_0\in\vv(P)\subseteq\re^n$,
  let $\Bb$ be the following subset of $\re$,
  \begin{equation}\label{eq6.3}
  \Bb=\left\{\beta\ge0:\exists r>0\ \ \exists C>0 : \  |\vec x-\vec x_0|<r\implies|P(\vec x)|^\beta\le C\left|Q(\vec x)\right|\right\}.\nonumber
 \end{equation}
 If $P\not\equiv0$ and there is a radius $\rho>0$ so that $\vv(Q)\cap\overline B_\rho(\vec x_0)\subseteq\vv(P)$, then:
  \begin{itemize}
    \item {\rm{(\L1)}} There exists some integer $N>0$ and some continuous semi-algebraic function $h:\overline B_\rho(\vec x_0)\to\re$ so that $P^N=h\cdot Q$.
    \item {\rm{(\L2)}} $\Bb$ is non-empty.
   \item {\rm{(\L3)}} The greatest lower bound, denoted $\beta_0=\glb\Bb$, satisfies $\beta_0\in\qu$. 
    \item {\rm{(\L4)}} $\beta_0\in\Bb$.
  \end{itemize}
\end{prop}
\begin{proof}[Sketch]
  It is straightforward to check that if $\beta_1\in\Bb$ and $\beta_2>\beta_1$, then $\beta_2\in\Bb$.  If $Q(\vec x_0)=0$ then $0\notin\Bb$.

  Statement (\L1) is a special case of Theorem 2.6.6 of \cite{bcr} (it holds more generally for continuous, semi-algebraic functions).  Statement (\L2) follows from (\L1), where $|h|$ has some maximum value on the closed ball with $r=\rho$ by continuity, so $N\in\Bb$.  The set $\Bb$ does not depend on $N$ from (\L 1) or the value of $\rho$.  However, $\Bb$ and its greatest lower bound $\beta_0\ge0$ depend on the point $\vec x_0$.  

  For statement (\L3), we refer to \cite{br}.  For the purposes of this article, we do not need to know the exact value of $\beta_0$, but for recent work on the problem of estimating $\beta_0$ we refer to \cite{bmn} and \cite{ouss}.

  Statement (\L4) seems to be the deepest part of the Proposition, and does not immediately follow from (\L1).  We refer to \cite{s}, where the claim is established for real analytic functions using geometric properties of subanalytic sets proved by \cite{p}.
\end{proof}
\begin{prop}[\L ojasiewicz Gradient Inequality]\label{prop6.2}
  Given a polynomial $P$ and a point $\vec x_0\in\vv(P)\subseteq\re^n$, if $P\not\equiv0$ and $\vec\nabla P(\vec x_0)=\vec0$ then there is a radius $\rho>0$ so that $\vv(|\vec\nabla P|^2)\cap\overline B_\rho(\vec x_0)\subseteq\vv(P)$, and:
  \begin{itemize}
    \item {\rm{(\L5)}} This set ${\mathbf\Theta}\subseteq\re$ is non-empty:
    \begin{equation}\label{eq6.4}
  {\mathbf\Theta}=\left\{\theta>0:\exists r>0\ \ \exists C>0 : \  |\vec x-\vec x_0|<r\implies|P(\vec x)|^\theta\le C\left|\vec\nabla P(\vec x)\right|\right\}.\nonumber
 \end{equation}
   \item {\rm{(\L6)}} $\theta_0=\glb{\mathbf\Theta}$ satisfies $\theta_0\in\qu\cap\left[\frac12,1\right)$. 
    \item {\rm{(\L7)}} $\theta_0\in{\mathbf\Theta}$.
  \end{itemize}
\end{prop}
\begin{proof}[Sketch]
  For the containment $\vv(|\vec\nabla P|^2)\cap\overline B_\rho(\vec x_0)\subseteq\vv(P)$ and statement (\L5) we refer to \cite{bm}, and then (\L7) follows from Proposition \ref{prop6.1} applied to the polynomial $\left|\vec\nabla P(\vec x)\right|^2$.  For estimates on $\theta_0$ as in (\L6), we refer to \cite{dk}, \cite{f}, and \cite{ouss1}.
\end{proof}
Recall that for a polynomial $P:\re^n\to\re$, the set of critical points ($\vec x_0$ such that $\vec\nabla P(\vec x_0)=\vec0$) is, for generic $P$, a finite set of isolated points in $\re^n$.   Proposition \ref{prop6.2} holds even for exceptional polynomials where the critical points are not isolated.

\section{Another property of rational functions}

The next Lemma uses (\L1) from Proposition \ref{prop6.1}, and considers an increase in the power of the denominator, unlike Lemma \ref{lem3.4} where the power decreased.  Lemma \ref{lem3.3} has a stronger hypothesis than Lemma \ref{lem3.4}; the rational function is not just locally bounded but approaches $0$ at points in its indeterminacy set. 


\begin{lem}\label{lem3.3}
  For an open ball $B_r(\vec x_0)=B_r$ and a rational function $P(\vec x)/Q(\vec x)$, if the function $s:B_r\to\re$ defined by
  $$s(\vec x)=\left\{\begin{array}{ll}{\displaystyle{\frac{P(\vec x)}{Q(\vec x)}}} & \mbox{\rm{for} $\vec x\in B_r\setminus\vv(Q)$}\\ & \\0& \mbox{\rm{for} $\vec x\in B_r\cap\vv(Q)$} \end{array}\right.$$ is continuous on $B_r$,
%
  then there exists a rational number $\varphi>0$ so that
    $$s_\varphi(\vec x)=\left\{\begin{array}{ll}{\displaystyle{\frac{|P(\vec x)|}{|Q(\vec x)|^{1+\varphi}}}} & \mbox{\rm{for} $\vec x\in B_r\setminus\vv(Q)$}\\ & \\0& \mbox{\rm{for} $\vec x\in B_r\cap\vv(Q)$} \end{array}\right.$$ is continuous on $B_r$.
%
\end{lem}
\begin{proof}
    $|P(\vec x)|/|Q(\vec x)|^{1+\varphi}$ is continuous at points outside $\vv(Q)$ for any $\varphi$, so we only need to find $\varphi$ so that $|P|/|Q|^{1+\varphi}$ approaches $0$ for points in $B_r\cap\vv(Q)$.  The continuous function $s$ is locally bounded near any point of $B_r$, so by Lemma \ref{lem2.5} applied to a neighborhood of any point in $B_r\cap\vv(Q)$ where $s$ is bounded, $B_r\cap\vv(Q)\subseteq B_r\cap\vv(P)$.
  
    The graph of $s(\vec x)$, $$G=\{(\vec x,s(\vec x)):\vec x\in B_r\}\subseteq B_r\times\re\subseteq\re^{n+1}$$ is contained in the real algebraic set $$\vv(x_{n+1}Q-P)=\{(x_1,\ldots,x_n,x_{n+1}):x_{n+1}Q(\vec x)-P(\vec x)=0\}\subseteq\re^{n+1}.$$ At each point $\vec x_1\in B_r$ where $Q(\vec x_1)=0$, we also have $P(\vec x_1)=0$, and $x_{n+1}Q(\vec x_1)-P(\vec x_1)=0$ for every $x_{n+1}\in\re$, so $\vv(x_{n+1}Q-P)$ contains a vertical line $\{(\vec x_1,x_{n+1}):x_{n+1}\in\re)\}$.  To get the graph of $s$,  remove the positive and negative parts of each such line, but keep the point $(\vec x_1,0)$, resulting in: $$G=((\vv(x_{n+1}Q-P)\setminus(\vv(Q)\times\re))\cup(\vv(Q)\times\{0\}))\cap(B_r\times\re).$$  This description shows that $G$ is a semi-algebraic set in $\re^{n+1}$ (after some elementary set operations, $G$ is a finite union of finite intersections as in Definition \ref{def4.2}).
    
    So, by construction, the function $s$ with graph $G$ is continuous on $B_r$ and semi-algebraic, and has a zero set that contains $B_r\cap\vv(Q)$.  These are exactly the hypotheses needed for   Theorem 2.6.6 of \cite{bcr} to apply to $s$ and $Q$, concluding that there exist some integer $\nu$ and some continuous, semi-algebraic function $h:B_r\to\re$ so that $(s(\vec x))^\nu=h(\vec x)\cdot Q(\vec x)$ for $\vec x\in B_r$.  (This is statement (\L1) from Proposition \ref{prop6.1}, in the continuous semi-algebraic case.)  For $\vec x\in B_r\setminus\vv(Q)$,
  \begin{eqnarray*}
      \left(\frac{P(\vec x)}{Q(\vec x)}\right)^\nu&=&h(\vec x)Q(\vec x)\\
      \implies\frac{(P(\vec x))^\nu|P(\vec x)|^{1/2}}{(Q(\vec x))^{\nu+1}}&=&h(\vec x)|P(\vec x)|^{1/2}\\
      \implies\frac{|P(\vec x)|^{2\nu+1}}{|Q(\vec x)|^{2\nu+2}}&=&h(\vec x)^2|P(\vec x)|\\
      \implies\frac{|P(\vec x)|}{|Q(\vec x)|^{(2\nu+2)/(2\nu+1)}}&=&|h(\vec x)^2P(\vec x)|^{1/(2\nu+1)}.
  \end{eqnarray*}
  This implies the claimed continuity, with $\varphi=\frac1{(2\nu+1)}$.  We remark that if $h$ is bounded by $c$ on some subset of $B_r$, then $|s|^\nu\le c|Q|$, as in (\L2) from Proposition \ref{prop6.1}.
\end{proof}
Lemma \ref{lem3.3} will be a key step in both Theorem \ref{lem2.9} and Theorem \ref{thm3.3}; its general idea is that the continuity of the extension $s$ is an ``open'' property in the sense that it is stable under small changes in the exponent of $|Q|$.  Example \ref{ex1.3} satisfies the assumptions of Lemma \ref{lem3.3} and shows that $\varphi$ may need to be small: $P/Q^{1.04}$ has a continuous extension on the disk $B_1(\vec0)$, but $P/Q^{1.05}$ does not.  
\begin{example}\label{ex2.8}
  The rational function $P/Q:\re^2\to\re$, for $P=xy$ and $Q=x$, extends to the polynomial function $y$.  Because $P/Q$ is bounded on any disk $B_r(\vec0)$, Lemma \ref{lem3.4} applies: $\frac{|P|}{|Q|^{1-\tau}}$ extends to the continuous function $|y||x|^\tau$ for any $\tau>0$.  Proposition \ref{lem2.7} also applies: $\frac{|P|}{|Q|^{\xi_0}}=|y|/|x|^{\xi_0-1}$ is unbounded for any $\xi_0>1$.  However, the extension $s(x,y)$ from Lemma \ref{lem3.3} is not the same as the extension $y$, and $P$ and $Q$ do not satisfy the assumption or the conclusion of Lemma \ref{lem3.3}.  This example also has the properties (1), (2), and (3) from Proposition \ref{prop2.4}, where the common factor $F=x$ has a non-singular point in its zero set and the indeterminacy locus has codimension $1$.
\end{example}
\begin{example}\label{ex2.9}
   Another trivial example, $\displaystyle{\frac PQ=\frac{x^2+y^2}{x^2+y^2}}$, shows that it is not enough in Lemma \ref{lem3.3} to require that $P/Q$ extend at one point to some continuous (or even constant) function; the extension $s$ from the Lemma must be both continuous and have value $0$ on the indeterminacy set.  Proposition \ref{prop2.4} also does not apply; the common factor $F=x^2+y^2$ has no nonsingular points and the codimension of the indeterminacy set is $2$.
\end{example}
\begin{example}\label{ex2.12}
  The rational function $P/Q:\re^3\to\re$, for $P=zx^2$ and $Q=x^2+y^2$, is bounded on any ball $B_r(\vec0)$, so Lemma \ref{lem3.4} and Proposition \ref{lem2.7} both apply. 
 The indeterminacy set is the $z$-axis, and $P$ and $Q$ have no non-constant common factor.  The extension $s(x,y,z)$ from Lemma \ref{lem3.3} satisfies ${\displaystyle{\lim_{(x,y,z)\to\vec0}s(x,y,z)=0}}$ but unlike Example \ref{ex2.8}, $P/Q$ does not have any continuous extension on any ball $B_r(\vec0)$: if $0<z_0<r$ then, $${\displaystyle{\lim_{(x,0,z)\to(0,0,z_0)}\frac{zx^2}{x^2+y^2}=z_0\ne\lim_{(0,y,z)\to(0,0,z_0)}\frac{zx^2}{x^2+y^2}=0}}.$$  This $P$ and $Q$ do not satisfy the assumption or the conclusion of Lemma \ref{lem3.3}.
\end{example}


\section{Existence of discontinuous, locally bounded rational functions}\label{sec2}

The proof of our main result about rational functions uses the Lemmas of the previous Sections, and the following Lemma, about two continuous functions and the minimum value of one when restricted to the level set of the other.

\begin{lem}\label{lem2.10}
  For a closed ball ${\overline B}_r(\vec x_0)\subseteq\re^n$ and continuous functions $\Phi:{\overline B}_r(\vec x_0)\to\re$, $\Upsilon:{\overline B}_r(\vec x_0)\to\re$, if there is a point $\vec x_1\in B_r(\vec x_0)$ so that $\Phi(\vec x_1)=\Upsilon(\vec x_1)=0$ is the minimum value for both $\Phi$ and $\Upsilon$, and there is no neighborhood of $\vec x_1$ on which $\Phi$ is constant, then  there is some $u_1>0$ so that the following function:
   $$\gamma(u)=\min\{\Upsilon(\vec x):\vec x\in {\overline B}_r(\vec x_0)\mbox{\ and\ }\Phi(\vec x)=u\}$$
   is well-defined and lower semicontinuous on the interval $[0,u_1]$, and satisfies $$\displaystyle{\lim_{u\to0^+}\gamma(u)=\gamma(0)=0}.$$
\end{lem}
\begin{proof}
  On the level set $\{\Phi=0\}$, $\Upsilon$ attains its minimum value $\gamma(0)=\Upsilon(x_0)=0$.
  For any $\epsilon_1>0$ there is some open ball $B_\delta(\vec x_1)$ so that if $\vec x\in B_\delta(\vec x_1)$ then $\Upsilon(\vec x)<\epsilon_1$.  By hypothesis, there is some point $\vec x_2\in B_\delta(\vec x_1)$ with $0<\Phi(\vec x_2)$.  By the Intermediate Value Theorem (applied to $\Phi$ on the segment from $\vec x_1$ to $\vec x_2$), for every $u\in[0,\Phi(\vec x_2)]$, the level set $\{\vec x\in {\overline B}_r(\vec x_0):\Phi(\vec x)=u\}$ meets the set $B_\delta(\vec x_1)$ in at least one point $\vec x_3$.  This level set is non-empty and compact, so the minimum value of $\Upsilon$ on the set, $\gamma(u)=\Upsilon(\vec x_4)$, exists (showing $\gamma$ is well-defined on $[0,u_1]$ with $u_1=\Phi(\vec x_2)$) and satisfies $0\le\gamma(u)\le\Upsilon(\vec x_3)<\epsilon_1$ (showing $\displaystyle{\lim_{u\to0^+}\gamma(u)=0}$).  The lower semicontinuity of $\gamma$ at a point $u_0$ in $[0,u_1]$ refers to the property that for any $\epsilon_2>0$ there is some neighborhood $W$ of $u_0$ in $[0,u_1]$ so that if $u\in W$ then $\gamma(u)>\gamma(u_0)-\epsilon_2$.  Let $V$ be the relatively open set $\{\vec x\in{\overline B}_r(\vec x_0):\Upsilon(\vec x)>\gamma(u_0)-\epsilon_2\}$. 
  The complement ${\overline B}_r(\vec x_0)\setminus V$ is compact, and its image $\Phi({\overline B}_r(\vec x_0)\setminus V)$ is compact in $\re$.  The complement $[0,u_1]\setminus(\Phi({\overline B}_r(\vec x_0)\setminus V))$ is the required neighborhood $W$ of $u_0$.  A point $u\in[0,u_1]$ is in $W$ if and only if there is no $\vec x\in {\overline B}_r(\vec x_0)\setminus V$ with $\Phi(\vec x)=u$; equivalently $V$ contains the level set $\{\vec x\in {\overline B}_r(\vec x_0):\Phi(\vec x)=u\}$, which is non-empty by the above Intermediate Value Theorem step.   The point $u_0\in W$ because every $\vec x$ in the level set $\{\Phi=u_0\}$ satisfies $\Upsilon(\vec x)\ge\gamma(u_0)$, so $\vec x\in V$.  For every other $u\in W$, the minimum value of $\Upsilon$ on the level set $\{\Phi=u\}\subseteq V$ satisfies $\gamma(u)>\gamma(u_0)-\epsilon_2$.
\end{proof}

The following Theorem on rational functions $P/Q$ is our main result.  We give two proofs --- the first proof (Case 1.) is similar to the argument in the Proof of Theorem 6.4 in \cite{bm}; it uses just the (\L1) version of the \L ojasiewicz inequality but we only prove the claim for $Q$ with an isolated zero.  The second proof (Case 2.) is shorter and works for any $Q$, but uses the stronger  statements (\L3) and (\L4).  The uniqueness part of the proof shows that the existence of exponents so that $P^{2p}/Q^{2e}$ is bounded and discontinuous implies (and so is equivalent to) the properties (\L3) and (\L4) of the polynomials $P$, $Q$.

\begin{thm}\label{lem2.9}
  For a rational function $P/Q$, suppose $P\not\equiv0$.   If there is some open set $\Delta$ with $\vec0\in\Delta\cap\vv(Q)\subseteq\Delta\cap\vv(P)$, then there exists some $r_1>0$ and some positive integers $p$, $e$ so that for any open ball $B_r(\vec0)$ with radius $0<r\le r_1$, the rational function $P^{2p}/Q^{2e}$ is bounded on $B_r$, and the function
    \begin{equation}\label{eq5.1}
    s(\vec x)=\left\{\begin{array}{ll}{\displaystyle{\frac{(P(\vec x))^{2p}}{(Q(\vec x))^{2e}}}} & \mbox{\rm{for} $\vec x\in B_r\setminus\vv(Q)$}\\ & \\0& \mbox{\rm{for} $\vec x\in B_r\cap\vv(Q)$} \end{array}\right.
    \end{equation} is not continuous on $B_r$.  
    The ratio $p/e$ is unique, in the sense that if $p_0$ and $q_0$ are positive integers, $P^{2p_0}/Q^{2q_0}$ is a bounded rational function on some $B_{r_2}(\vec x_0)$, and its extension as in \mbox{\rm{(\ref{eq5.1})}} is discontinuous on $B_r$ for all $r\in(0,r_2]$, then $p_0/q_0=p/e$.
\end{thm}
\begin{proof}
  Case 1.  $\vec 0$ is an isolated point of $\vv(Q)$, so that there is some open set $\Delta$ with $\{\vec0\}=\Delta\cap\vv(Q)\subseteq\Delta\cap\vv(P)$

  Let $B_R(\vec0)$ be some open ball with closure contained in $\Delta$.
    Consider the map $\Gamma:\re^n\to\re^{2}$ defined by $$\Gamma(\vec x)=((P(\vec x))^2,(Q(\vec x))^2).$$     The closed ball ${\overline B}_R$ has image $\Gamma({\overline B}_R)$ in the $(u,v)$ target plane, and the image has the following properties: $\Gamma({\overline B}_R)$ is a compact subset of $\re^2$, it is a subset of the closed first quadrant, it is a semi-algebraic set (\cite{bcr} Prop.\ 2.2.7), it contains the origin $(P^2(\vec0),Q^2(\vec0))=(0,0)$ as well as infinitely many other points, so it has dimension $1$ or $2$, and because ${\overline B}_R\cap\vv(Q)\subseteq{\overline B}_R\cap\vv(P)$, there are no points of the form $(u,0)$ in the image with $u\ne0$.  
    $\Gamma({\overline B}_R)$ contains its boundary set, denoted $\partial\Gamma({\overline B}_R)$.

  Lemma \ref{lem2.10} applies to $\Phi=P^2$ and $\Upsilon=Q^2$ on ${\overline B}_R$, so that there is some interval $[0,u_1]$ and a function $\gamma:[0,u_1]\to\re$ defined by: \begin{equation}\label{eq10}
      \gamma(u)=\min\{Q^2(\vec x):\vec x\in{\overline B}_R\mbox{\ and\ }(P(\vec x))^2=u\}.
    \end{equation}
    By Lemma \ref{lem2.10}, $\displaystyle{\lim_{u\to0^+}\gamma(u)=0}$.  By construction, each point $(u,v)$ on the graph of $\gamma$ is in the image $\Gamma({\overline B}_R)$; either $(u,v)=(0,0)$, or $u>0$ and $v=\gamma(u)>0$ is the minimum value $(Q(\vec x_2))^2$ on the set $\{P^2=u\}$, for some $\vec x_2\in{\overline B}_R$ with $(P(\vec x_2))^2=u$ so $(u,v)=\Gamma(\vec x_2)$.  Also, $(u,v)=((P(\vec x_2))^2,(Q(\vec x_2))^2)$ is a boundary point in the image because $v$ is a minimum value of $Q^2$.  This shows that the intersection of the open first quadrant and $\partial\Gamma({\overline B}_R)$ is infinite and one-dimensional.

    Proposition \ref{lem2.11} applies to $\partial\Gamma({\overline B}_R)$.  There is a small box $(0,\delta)\times(0,\eta)$ and there are $K\ge1$ functions so that $\partial\Gamma({\overline B}_R)\cap(0,\delta)\times(0,\eta)$ is the union of $K$ disjoint graphs of continuous functions.  By the Intermediate Value Theorem, one of these functions, which can be labeled $b_1$, satisfies $b_1(u)\le b_k(u)$ for all $0<u<\delta$.  The values of $\gamma(u)$ are also the lowest $v$-values of points on $\partial\Gamma({\overline B}_R)$, so we can conclude that for $0<u<u_0=\min\{u_1,\delta\}$, $\gamma(u)=b_1(u)$.  In particular, although $\gamma$ may be discontinuous on some larger interval, because $\gamma$ coincides with $b_1$, $\gamma$ is continuous on $[0,u_0)$ and real analytic on $(0,u_0)$ with $e=e_1$ as in Proposition \ref{lem2.11} and a Puiseux expansion $$\gamma(u)=\sum_{j=p}^{\infty}\gamma_ju^{j/e}=\gamma_pu^{p/e}+\gamma_{p+1}u^{(p+1)/e}+\cdots$$ for some lowest degree $p>0$ with $\gamma_p\ne0$.

    The composite $(\gamma(u))^{e/p}$ is increasing on $[0,u_0)$, with derivative on $(0,u_0)$:
    \begin{eqnarray*}
        &&\frac ep(\gamma(u))^{(e/p)-1}\gamma^\prime(u)\\
        &=&\frac ep\left(\left(\gamma_pu^{p/e}+\gamma_{p+1}u^{(p+1)/e}+\cdots\right)^{e-p}\left(\frac pe\gamma_p u^{(p/e)-1}+\cdots\right)^p\right)^{1/p}.\\
    \end{eqnarray*}
    In particular, $\displaystyle{\lim_{u\to0^+}\frac d{du}\left((\gamma(u))^{e/p}\right)=\gamma_p^{e/p}>0}$.
    So if we pick some constants $0<K_1<\gamma_p^{e/p}<K_2$, then there is some $\delta^\prime>0$ so that for $0<u<\delta^\prime$, $$K_1u<(\gamma(u))^{e/p}<K_2u.$$  Using the continuity of $P$, corresponding to $\delta^\prime$ there is some ball $B_{r_1}$ as claimed so that $\overline{B}_{r_1}\subseteq\Delta_2$ and if $\vec x\in B_{r_1}$ then $(P(\vec x))^2<\delta^\prime$.  So for such $\vec x$ with $0<u=(P(\vec x))^2<\delta^\prime$, the functions $P$ and $Q$ satisfy:
    \begin{eqnarray*}
      \frac{(P(\vec x))^{2p}}{(Q(\vec x))^{2e}}&=&\left(\frac{(P(\vec x))^2}{(Q(\vec x))^{2e/p}}\right)^p\le\left(\frac u{(\gamma(u))^{e/p}}\right)^p<\left(\frac u{K_1u}\right)^p=K_1^{-p}.
\end{eqnarray*}
    If $P(\vec x)=0$ then $P^{2p}/Q^{2e}$ is either $0$ or undefined, so we can conclude $P^{2p}/Q^{2e}$ is bounded on the intersection of its domain with $B_{r_1}$.  Up to this point the Proof has followed the construction of the Proof of Theorem 6.4 in \cite{bm}, and has used only ${\overline B}_R\cap\vv(Q)\subseteq{\overline B}_R\cap\vv(P)$.

    For any number $\zeta>e/p$, points $(u,\gamma(u))$ with $0<u<\delta^\prime$ satisfy:
    $$\frac{u}{(\gamma(u))^\zeta}>\frac u{((K_2u)^{p/e})^\zeta}=\frac1{K_2^{\zeta p/e}}\frac1{u^{\zeta p/e-1}}\implies \lim_{u\to0^+}\frac u{(\gamma(u))^\zeta}=+\infty.$$

    Toward drawing a conclusion about $P$ and $Q$ near $\vec0\in\re^n$, choose any sequence $u_k\to0^+$, so that then $\gamma(u_k)\to0^+$ and $(u_k,\gamma(u_k))$ converges to $(0,0)$ in $\re^2$.  We can then consider an inverse image sequence in ${\overline B}_R$ by choosing, for each $k$, a point $\vec w_k$ in $\Gamma^{-1}(\{(u_k,\gamma(u_k))\})\cap{\overline B}_R$.  Then $(P(\vec w_k))^2=u_k$, $(Q(\vec w_k))^2=\gamma(u_k)$, and \begin{equation}\label{eq2.1}
      \frac{(P(\vec w_k))^2}{(Q(\vec w_k))^{2\zeta}}=\frac{u_k}{(\gamma(u_k))^{\zeta}}\to+\infty,
    \end{equation}
    so $P^{2p}/Q^{2p\zeta}$ is unbounded on ${\overline B}_R$.  However, while such a sequence $\vec w_k\in{\overline B}_R$ is bounded, it does not follow that any of the points $\vec w_k$ is close to $\vec0\in\re^n$ and we cannot conclude anything about $P$, $Q$ on smaller neighborhoods of the origin.

    We remark that everything so far in this construction depends on the initial choice of the ball $B_R$ in $\re^n$.  If $B_R$ is replaced by a smaller ball, then the image $\Gamma(\overline{B}_R)$ is replaced by some subset, and some values $\gamma(u)$ could increase (being defined as the minimum over a smaller set).  The above calculations do not rule out the possibility that the exponents $e$ and $p$, or the constants $K_1$, $K_2$, $\delta^\prime$, $r_1$, may change depending on $R$.  In general, the properties of the image of a small neighborhood under polynomial maps can depend strongly on the size of the neighborhood (see \cite{jt}).

    So, keeping the initially chosen $B_R$ and the subsequently constructed $B_{r_1}$ both fixed, we now start using the hypothesis that the zero set of $Q$ in ${\overline B}_R$ is the isolated point $\{\vec0\}$.  Statement (\L1) from Proposition \ref{prop6.1} applies to the polynomials $|\vec x|^2$ and $Q(\vec x)$  ---  there exists a rational number $L>0$ and real number $c>0$ so that for all $\vec x\in\overline{B}_R$, $|\vec x|^L\le c|Q(\vec x)|$.  The points $\vec w_k$ in the above sequence satisfy $|\vec w_k|^L\le c|Q(\vec w_k)|$, and $$|\vec w_k|\le c^{1/L}((Q(\vec w_k))^2)^{1/(2L)}=c^{1/L}(\gamma(u_k))^{1/(2L)},$$ and so $\displaystyle{\lim_{k\to\infty}\vec w_k=\vec0}$.

    
    
    Because as $k\to\infty$, LHS (\ref{eq2.1}) approaches $\infty$ along the sequence $\vec w_k\to\vec0$, we can conclude $P^{2p}/|Q|^{2q^\prime}$ is unbounded on any neighborhood $B_r$ for any real exponent $q^\prime=\zeta p>e$.  It then follows from Lemma \ref{lem3.3} (which also used (\L1)) that  $P^{2p}/Q^{2e}$ cannot be extended to a continuous function of the form $s(\vec x)$ with $s(\vec0)=0$ as claimed.

Case 2.  From Proposition \ref{prop6.1} statements (\L3) and (\L4), there exists a minimal rational exponent $0<\beta_0\in\Bb$, so that there exist $r_1>0$ and $C>0$ with 
\begin{equation}\label{eq5.5}
  |P(\vec x)|^{\beta_0}\le C\left|Q(\vec x)\right|
\end{equation}
on $B_{r_1}$.  Let $p$ and $e$ be any positive integers with $\beta_0=p/e$, then: $|P(\vec x)|^{2p}\le C^{2e}\left|Q(\vec x)\right|^{2e}$.  This shows $s(\vec x)$ is bounded as claimed.

Suppose toward a contradiction that $s(\vec x)$ is continuous on some ball $B_r$ with $r\le r_1$.  Then Lemma \ref{lem3.3} (using (\L1)) applies: there is some rational $\varphi>0$ so that 
$$s_{\varphi}(\vec x)=\left\{\begin{array}{ll}{\displaystyle{\frac{(P(\vec x))^{2p}}{(Q^{2e})^{1+\varphi}}}} & \mbox{\rm{for} $Q(\vec x)\ne0$}\\ & \\0& \mbox{\rm{for} $Q(\vec x)=0$} \end{array}\right.$$
is continuous, and bounded by some constant $C_\varphi>0$ on some possibly smaller ball $B_{r^\prime}$.
It follows that for $\vec x\in B_{r^\prime}\setminus\vv(Q)$, 
\begin{equation}\label{eq5.6}
|P(\vec x)|^{2p/(2e(1+\varphi))}\le C_\varphi^{1/(2e(1+\varphi))}|Q(\vec x)|.
\end{equation}
If $Q(\vec x)=0$ then both sides of the inequality (\ref{eq5.6}) are $=0$ by the inequality (\ref{eq5.5}), so the inequality (\ref{eq5.6}) holds for all $\vec x\in B_{r^\prime}$.  This contradicts the minimality of $\beta_0$.

The following proof of uniqueness refers to the set $\Bb$ from Proposition \ref{prop6.1}, but does not use any of the claims (\L1) -- (\L4) from Proposition \ref{prop6.1}, and in fact shows that (\L3) and (\L4) for polynomials $P$ and $Q$ follow from the existence of a bounded, discontinuous rational function of the form $P^{2p_0}/Q^{2q_0}$.

If $P^{2p_0}/Q^{2q_0}$ is bounded by some constant $C$ on some ball $B_{r_2}(\vec x_0)$ then by Lemma \ref{lem2.5}, $B_{r_2}\cap\vv(Q)\subseteq B_{r_2}\cap\vv(P)$, and for $\vec x\in B_{r_2}$, $P^{2p_0}\le CQ^{2q_0}$, and the fraction $p_0/q_0$ is in the set $\Bb\subseteq[0,\infty)$.  Let $\beta_0$ be the real number $\glb\Bb$, so $\beta_0\le p_0/q_0$.  Now assume $P^{2p_0}/Q^{2q_0}$ has an extension $s_0$ as in line (\ref{eq5.1}) which is discontinuous on $B_r$ for any $0<r\le r_2$, and suppose toward a contradiction that $p_0/q_0>\beta_0$.  Then there is some real number $\beta_1\in\Bb$ with $0\le\beta_0\le\beta_1<p_0/q_0$ (otherwise $p_0/q_0$ would be the $\glb$) and there are some integers $p_1$, $q_1$ so that $\beta_1<p_1/q_1<p_0/q_0$ and $p_1/q_1\in\Bb$.  It follows that $P^{2p_1}/Q^{2p_1}$ is bounded on some $B_{r_3}(\vec x_0)$ and by Lemma \ref{lem3.4}, for any $\tau>0$, the function   $$S(\vec x)=\left\{\begin{array}{ll}{\displaystyle{\frac{(P(\vec x))^{2p_1}}{((Q(\vec x))^{2q_1})^{1-\tau}}}} & \mbox{\rm{for} $\vec x\in B_{r_3}\setminus\vv(Q)$}\\ & \\0& \mbox{\rm{for} $\vec x\in B_{r_3}\cap\vv(Q)$}\end{array}\right.$$
 is continuous on $B_{r_3}$.  For $\tau=1-\frac{p_1/q_1}{p_0/q_0}>0$, the function $(S(\vec x))^{p_0/p_1}$ is continuous on $B_{r_3}$ and, for $r_4=\min\{r_2,r_3\}$ is equal to $s(\vec x)$ on $B_{r_4}$, contradicting the discontinuity of $s$.
\end{proof}
\begin{example}\label{ex3.5}
  The polynomials $P=x^{9}y^3$ and $Q=x^{10}+y^2$ satisfy the hypotheses of Theorem \ref{lem2.9} (in Case 1.\ with $\vv(Q)=\{\vec0\}$) and the rational function $P/Q$ extends continuously to the origin.  The conclusion of Theorem \ref{lem2.9} is that there are some integer powers $p$, $q$ so that  $P^{2p}/Q^{2e}$ is a locally bounded rational function that does not extend continuously with limit $0$ at the origin.  In the same way as Example \ref{ex1.2}, along the curve $y=x^5$ the denominator is equal to $2x^{10}$.  Raising $P(x,x^5)=x^{9}x^{15}=x^{24}$ to the $\frac{10}{24}$ power would make the degrees equal, so that the (non-rational) function $\displaystyle{\frac{(x^{9}y^3)^{10/24}}{x^{10}+y^2}}$ approaches $\frac12$ along $y=\pm x^5$ but approaches $0$ along the $x$-axis and $y$-axis.  A locally bounded rational function as in Theorem \ref{lem2.9} with a discontinuous extension $s(x,y)$ is $\displaystyle{\frac{(x^{9}y^3)^{10}}{(x^{10}+y^2)^{24}}}$.
\end{example}
\begin{example}\label{ex3.4}
  Suppose $P\not\equiv0$ and $Q\not\equiv0$ are both homogeneous polynomials: there are integers $\mu\ge1$, $\nu\ge1$ so that for all $\lambda\in\re$, $\vec x\in\re^n$, $P(\lambda\vec x)=\lambda^\mu P(\vec x)$ and $Q(\lambda\vec x)=\lambda^\nu Q(\vec x)$.  Let $\kappa$ be the least common multiple of $\mu$, $\nu$, so that $\kappa=\mu \nu^\prime=\nu\mu^\prime$.  If $Q$ has an isolated zero at the origin, then for any $\vec x\ne\vec0$, the following rational function is constant on sets of the form $\{\lambda\vec x:\lambda\ne0\}$:
  \begin{equation}\label{eq3.1}
    \frac{(P(\lambda\vec x))^{\nu^\prime}}{(Q(\lambda\vec x))^{\mu^\prime}}=\frac{\lambda^{\mu\nu^\prime}(P(\vec x))^{\nu^\prime}}{\lambda^{\mu^\prime\nu}(Q(\vec x))^{\mu^\prime}}=\frac{(P(\vec x))^{\nu^\prime}}{(Q(\vec x))^{\mu^\prime}}.
  \end{equation}
  Further, ${(P(\vec x))^{\nu^\prime}}/{(Q(\vec x))^{\mu^\prime}}$ is continuous on the unit sphere, achieves some maximum and minimum values on the sphere, and is bounded on the interior.  If the function is non-constant on the sphere, then the expression (\ref{eq3.1}) will have different $\lambda\to0$ limits along two lines through the origin and ${(P(\vec x))^{\nu^\prime}}/{(Q(\vec x))^{\mu^\prime}}$ cannot be extended continuously to the origin.  The graph will exhibit these lines at constant heights as in the ruled surface from Example \ref{ex1.1}.  If the function is constant on the sphere: $(P(\vec x))^{\nu^\prime}/(Q(\vec x))^{\mu^\prime}\equiv c$ (as in Example \ref{ex2.9}), then it is constant on $\re^n\setminus\{\vec0\}$, and $c=0$ would contradict $P\not\equiv0$, so the expression (\ref{eq3.1}) has limit $c\ne0$ as $\lambda\to0$, but the function $s(\vec x)$ from Theorem \ref{lem2.9} is discontinuous.  
\end{example}


\section{Some complex constructions}\label{sec4}

In this Section we show how to use complex polynomials to construct bounded real rational functions.  We consider first one complex variable $z=x+iy$ in $\co=\re^2$ and then several variables $z_1,\ldots,z_m$ in $\co^m=\re^{2m}$.

\begin{example}\label{ex3.6}
  Real valued rational functions in two real variables $(x,y)$ can be constructed in the following way, which, like Example \ref{ex3.4}, is consistent with, but does not appeal to, Theorem \ref{lem2.9}.  For $z=x+iy$, let $f(z)$ and $g(z)$ be non-constant, holomorphic polynomials that both vanish at the origin $0+0i$.  For both $f$ and $g$, the zero at the origin is isolated, and they factor as $f(z)=z^\mu f_0(z)$ and $g(z)=z^\nu g_0(z)$ for multiplicities $\mu>0$, $\nu>0$ and polynomials $f_0$, $g_0$ with non-zero constant terms.  Let $\kappa=\mu \nu^\prime=\nu\mu^\prime$ as in Example \ref{ex3.4}.  The polynomial $f(z)$ and the complex conjugate $\overline{g(z)}$ satisfy, for $z$ in some deleted disk $B_R(0+0i)\setminus\{0+0i\}$ where $g(z)\ne0+0i$,
  \begin{eqnarray}
      \frac{(f(z))^{\nu^\prime}}{\left(\overline{g(z)}\right)^{\mu^\prime}}&=&\frac{z^{\mu\nu^\prime}(f_0(z))^{\nu^\prime}}{\z^{{\mu^\prime}\nu}\left(\overline{g_0(z)}\right)^{\mu^\prime}}=\left(\frac{z}{\z}\right)^{\kappa}\frac{(f_0(z))^{\nu^\prime}(g_0(z))^{\mu^\prime}}{(g_0(z))^{\mu^\prime}\left(\overline{g_0(z)}\right)^{\mu^\prime}}\nonumber\\
      &=&\left(\frac{z^2}{|z|^2}\right)^{\kappa}\frac{(f_0(z))^{\nu^\prime}(g_0(z))^{\mu^\prime}}{|g_0(z)|^{2{\mu^\prime}}}\label{eq4.12}\\
      &=&\left(\frac{z^2}{|z|^2}\right)^{\kappa}(c_0+c(z,\z)).\label{eq4.13}
  \end{eqnarray}
  The factor $f_0^{\nu^\prime} g_0^{\mu^\prime}/|g_0|^{2{\mu^\prime}}$ on line (\ref{eq4.12}) is non-vanishing and real analytic on $B_R$, so it is of the form $c_0+c(z,\z)$ with $0\ne c_0\in\co$ and $\displaystyle{\lim_{z\to0}c(z,\z)=0}$ as on line (\ref{eq4.13}).  The factor $z^{2\kappa}/|z|^{2\kappa}$ does not have a $z\to0$ limit at the origin for any $\mu$, $\nu$, but it is bounded (with norm $\equiv1$) so its product with $c(z,\z)$ has limit $0$.  We want to show that both the real and imaginary parts of (\ref{eq4.12}) have different limits as $z\to0$ along different real directions, and it is enough to show this holds for $c_0z^{2\kappa}/|z|^{2\kappa}$.  Let $c_0=R_0e^{i\phi_0}$ with $R_0>0$, and $z=re^{i\theta}$ for small $r>0$, so $$c_0z^{2\kappa}/|z|^{2\kappa}=R_0e^{i(2\theta\kappa+\phi_0)}=R_0\cos(2\theta\kappa+\phi_0)+iR_0\sin(2\theta\kappa+\phi_0).$$ 
  For $\theta_1=-\phi_0/(2\kappa)$ and $\theta_2=(\pi-\phi_0)/(2\kappa)$, $$\lim_{r\to0^+}\rp\left(c_0\frac{(re^{i\theta_1})^{2\kappa}}{|re^{i\theta_1}|^{2\kappa}}\right)=R_0+0i\ne\lim_{r\to0^+}\rp\left(c_0\frac{(re^{i\theta_2})^{2\kappa}}{|re^{i\theta_2}|^{2\kappa}}\right)=-R_0+0i.$$
  For $\theta_3=(\frac{\pi}2-\phi_0)/(2\kappa)$ and $\theta_4=(\frac{3\pi}2-\phi_0)/(2\kappa)$, $$\lim_{r\to0^+}\ip\left(c_0\frac{(re^{i\theta_3})^{2\kappa}}{|re^{i\theta_3}|^{2\kappa}}\right)=R_0+0i\ne\lim_{r\to0^+}\ip\left(c_0\frac{(re^{i\theta_4})^{2\kappa}}{|re^{i\theta_4}|^{2\kappa}}\right)=-R_0+0i.$$
  We can conclude that $f^{\nu^\prime}/\overline{g}^{\mu^\prime}$ is locally bounded on $B_R$ but both its real part $\rp(f^{\nu^\prime}/\overline{g}^{\mu^\prime})$ and its imaginary part $\ip(f^{\nu^\prime}/\overline{g}^{\mu^\prime})$,
 \begin{eqnarray*}
      \frac{(f(z))^{\nu^\prime}}{\left(\overline{g(z)}\right)^{\mu^\prime}}&=&\frac{\rp\left((f(z))^{\nu^\prime}(g(z))^{\mu^\prime}\right)}{|g(z)|^{2{\mu^\prime}}}+i\frac{\ip\left((f(z))^{\nu^\prime}(g(z))^{\mu^\prime}\right)}{|g(z)|^{2{\mu^\prime}}},
  \end{eqnarray*}
  give examples of a bounded real rational function, with an isolated zero in the denominator, that cannot extend continuously to the origin.  Example \ref{ex1.1} is exactly the case $\ip(z/(2\z))$.  Of course, not every real rational function occurs in this way.
\end{example}
 
\begin{example}\label{ex4.2}
  Let $h(z_1,\ldots,z_m)$ be a non-constant holomorphic polynomial in several variables.  Its zero set $\{\vec z:h(\vec z)=0+0i\}$ is always a non-empty real algebraic set in $\co^m=\re^{2m}$ with real dimension exactly $2m-2$.  The expression 
  \begin{equation}\label{eq4.10}
  \frac{h(z_1,\ldots,z_m)}{\left(\overline{h(z_1,\ldots,z_m)}\right)}=\frac{h^2}{|h|^2}=\frac{\rp(h^2)}{|h|^2}+i\frac{\ip(h^2)}{|h|^2}
  \end{equation}
  is no longer holomorphic, but it is bounded on $\co^m\setminus\vv(|h|^2)$ because its absolute value is $\equiv1$.  Its real and imaginary parts are bounded real rational functions on $\re^{2m}$, with $\vv(|h|^2)\subseteq\vv(\rp(h^2))$, and similarly for $\ip(h^2)$, so $\vv(|h|^2)$ is the indeterminacy set for both.  For $\vec z_0\in\co^m$ with $h(\vec z_0)=0+0i$, let $\ell(z)=\vec z_0+(c_1z,c_2z,\ldots,c_mz)$ be a complex line through $\vec z_0$ parametrized by $z$.  Because $h\not\equiv0$, there are some complex coefficients $c_1,\ldots,c_m$ so that $h(\ell(z))\not\equiv0$.  The calculations of Example \ref{ex3.6} apply to the single variable holomorphic polynomial $h\circ\ell=f=g$, with $\mu=\nu=\kappa$ and $\mu^\prime=\nu^\prime=1$, to find unequal limits of the real and imaginary parts of $h/\bar h$ along real lines contained in the image of $\ell$.  The conclusion is that $\rp(h^2)/|h|^2$ is a bounded rational function that does not extend continuously on any neighborhood of $\vec z_0$.  Further, because the exponents $\mu^\prime=\nu^\prime=1$ do not depend on $\vec z_0$, $\rp(h^2)/|h|^2$ cannot be extended continuously in any neighborhood of any point in $\vv(|h|^2)$, and similarly for $\ip(h^2)/|h|^2$.
\end{example}


\begin{example}\label{ex4.3}
For $h = z_1^2 - z_2^3$, with  $z_1 = x_1 + i y_1$, $z_2 = x_2 + i y_2$, the construction of Example \ref{ex4.2} gives these two rational functions, each of which is bounded on $\re^4$ with indeterminacy set equal to the two-dimensional complex curve $\{z_1^2-z_2^3=0+0i\}$:
\begin{eqnarray*}
 \rp(h/\bar h)&=&\frac{(x_1^2 - y_1^2 - x_2^3 + 3x_2 y_2^2)^2 - (2x_1 y_1 - 3x_2^2 y_2 + y_2^3)^2}{(x_1^2 - y_1^2 - x_2^3 + 3x_2 y_2^2)^2 + (2x_1 y_1 - 3x_2^2 y_2 + y_2^3)^2},\\
  \ip(h/\bar h)&=&\frac{2(x_1^2 - y_1^2 - x_2^3 + 3x_2 y_2^2)(2x_1 y_1 - 3x_2^2 y_2 + y_2^3)}{(x_1^2 - y_1^2 - x_2^3 + 3x_2 y_2^2)^2 + (2x_1 y_1 - 3x_2^2 y_2 + y_2^3)^2}.
\end{eqnarray*}
\end{example}
 
\section{Using the \L ojasiewicz gradient inequality}\label{sec3}

Recall from Proposition \ref{prop2.4} that when the domain is $2$-dimensional, a locally bounded real rational function $P/Q$  will either have only isolated points in the zero set $\vv(Q)$ (Examples \ref{ex1.1}, \ref{ex1.2}, \ref{ex1.3}), or a non-constant common factor $P=M_1F$, $Q=M_2F$ (Example \ref{ex2.8}), or both (Example \ref{ex2.9}).

For higher dimensions $n$, real polynomials $Q$ with isolated zeros are exceptional cases and a non-empty zero set generically has positive dimension.  This Section considers the case where $Q$ is the norm of the gradient of $P$.  Because the gradient generally has isolated zeroes, this case gives a method to construct many discontinuous, locally bounded rational functions of the form (\ref{eq2}) with an isolated zero in the denominator.  The main result, Theorem \ref{thm3.3}, applies even when the gradient has a non-isolated zero.

\begin{example}\label{ex3.2}
  Start with a polynomial $P$ with critical point $\vec x_0$ as in Proposition \ref{prop6.2}, so $P(\vec x_0)=|\vec\nabla P(\vec x_0)|=0$.  If $P$ is non-constant then the squared norm $|\vec\nabla P(\vec x)|^2$ is a polynomial $\not\equiv0$.  Take any particular number $\theta$ in the set $\mathbf\Theta$ from statement (\L5) from Proposition \ref{prop6.2}, so there is some ball $B_r(\vec x_0)$, where $\vv(|\vec\nabla P|^2)\cap B_r\subseteq\vv(P)\cap B_r$.  For $\vec x\in B_r\setminus\vv(|\vec\nabla P|^2)$ and the bound $C$ corresponding to $\theta$ from Proposition \ref{prop6.2},
\begin{equation*}
  0\le\frac{|P(\vec x)|^{2\theta}}{\left|\vec\nabla P(\vec x)\right|^2}\le C^2.
\end{equation*}
If we also assume there exist positive integers $p$, $q$, so that $p/q=\theta$, then raising to the power of $q$ gives:
\begin{equation}\label{eq2}
  0\le\frac{(P(\vec x))^{2p}}{\left|\vec\nabla P(\vec x)\right|^{2q}}\le C^{2q}.
\end{equation}
The rational function (\ref{eq2}) so constructed is bounded on $B_r$; the intersection of its indeterminacy set and $B_r$ is $\vv(|\vec\nabla P|^2)\cap B_r$.  
\end{example}
In the construction of Example \ref{ex3.2}, the denominator of (\ref{eq2}) is determined by the choice of numerator and the exponents $p$, $q$, and the denominator may have a non-trivial factor in common with the numerator, as in the following Notation.
\begin{notation}\label{not4.5}
  Given $P$, $\vec x_0$, $p$, $q$ as in the construction of Example \ref{ex3.2}, the numerator and denominator of the expression (\ref{eq2}) each have a unique factorization (up to reordering and constant multiples, and not depending on $r$) into irreducible polynomials.  Let $P^{2p}=F\cdot N$ and $\left|\vec\nabla P(\vec x)\right|^{2q}=F\cdot D$, so that $F$ is a product of all of the common factors.  If there are no common factors, then let $F=1$; if   $P^{2p}$ has $\left|\vec\nabla P(\vec x)\right|^{2q}$ as a factor, then let $D=1$.  Canceling $F$ from the numerator and denominator of the rational function (\ref{eq2}) gives  another rational function $\displaystyle{\frac ND}$, the {\underline{reduced}} form of $\displaystyle{\frac{(P(\vec x))^{2p}}{\left|\vec\nabla P(\vec x)\right|^{2q}}}$. 
\end{notation}

After such a cancellation, so that the polynomials $N$ and $D$ have no common factors and are both $\not\equiv0$, the reduced form has a possibly smaller indeterminacy set, and satisfies, for  $\vec x\in B_r\setminus\vv(|\vec\nabla P|^2)$:
$$0\le\frac{(P(\vec x))^{2p}}{\left|\vec\nabla P(\vec x)\right|^{2q}}=\frac{N(\vec x)}{D(\vec x)}\le C^{2q}.$$


\begin{thm}\label{thm3.2}
  Given a polynomial $P\not\equiv0$ with $P(\vec x_0)=|\vec\nabla P(\vec x_0)|=0$, positive integers $p$, $q$ so that  $p/q=\theta\in{\mathbf\Theta}$, and $r$, $C$, $N$ and $D$ as constructed in Example {\rm{\ref{ex3.2}}} and Notation {\rm{\ref{not4.5}}}, the rational function $N/D$ is bounded on $B_r$ and $\vv(D)\cap B_r$ has dimension $\le n-2$. 
\end{thm}
\begin{proof}
  The rational function $P^{2p}/|\vec\nabla P|^{2q}$ is bounded as in line (\ref{eq2}); considering that $\vv(D)\subseteq\vv(|\vec\nabla P|^{2q})$, we just need to check that its reduced form $N/D$ is bounded on its possibly larger domain, as in Definition \ref{def2.4}.  At any point $\vec x\in B_r$ where $D(\vec x)\ne0$, as in the Proof of Lemma \ref{lem2.5} there is some sequence $\vec x_i\in B_r\setminus\vv(|\vec\nabla P|^{2q})$ converging to $\vec x$.   By the continuity of $N/D$ at $\vec x$, the limit (\ref{eq4}) exists, and it is equal to the limit (\ref{eq7}) by the choice of the $\vec x_i$ sequence, and it has the same upper bound as (\ref{eq2}):
\begin{eqnarray}
  \frac{N(\vec x)}{D(\vec x)}&=&\lim_{i\to\infty}\frac{N(\vec x_i)}{D(\vec x_i)}\label{eq4}\\
  &=&\lim_{i\to\infty}\frac{(P(\vec x_i))^{2p}}{\left|\vec\nabla P(\vec x_i)\right|^{2q}}\le C^{2q}.\label{eq7}
\end{eqnarray}
Similarly, $N(\vec x)/D(\vec x)\ge0$.

By statement $(2)$ from Proposition \ref{prop2.4}, because all common factors have been canceled (including any factor $F$ as in statement $(1)$, and possibly also other factors) the indeterminacy set of $N/D$ does not contain any point (anywhere in $\re^n$) with a neighborhood where the set is a submanifold with dimension $n-1$.  By Lemma \ref{lem2.5}, $\vv(D)\cap B_r\subseteq\vv(N)\cap B_r$, and so the indeterminacy set of $N/D$ intersects $B_r$ in $\vv(D)\cap\vv(N)\cap B_r=\vv(D)\cap B_r$.
 
\end{proof}


\begin{thm}\label{thm3.3}
  Given a polynomial $P\not\equiv0$ with $P(\vec x_0)=|\vec\nabla P(\vec x_0)|=0$, positive integers $p$, $q$ so that  $p/q=\theta\in{\mathbf\Theta}$, and $r$, $C$, $N$ and $D$ as constructed in Example {\rm{\ref{ex3.2}}} and Notation {\rm{\ref{not4.5}}}, define these functions $\re^n\to\re$,
    $$f(\vec x)=\left\{\begin{array}{ll}{\displaystyle{\frac{N(\vec x)}{D(\vec x)}}} & \mbox{\rm{for} $\vec x\not\in\vv(D)$}\\ & \\0& \mbox{\rm{for} $\vec x\in\vv(D)$} \end{array}\right.$$
and 
$$g(\vec x)=\left\{\begin{array}{ll}{\displaystyle{\frac{(P(\vec x))^{2p}}{|\vec\nabla P(\vec x)|^{2q}}}} & \mbox{\rm{for} $\vec\nabla P(\vec x)\ne\vec0$}\\ & \\0& \mbox{\rm{for} $\vec\nabla P(\vec x)=\vec0$} \end{array}\right..$$
The functions $f$ and $g$ are bounded on $B_{r}(\vec x_0)$, and further,
\begin{enumerate}
  \item If $p/q>\theta_0$, then $f(\vec x)$ and $g(\vec x)$ are equal and continuous on $B_{r^\prime}(\vec x_0)$ for some $r^\prime\le r$.

  \item If $p/q=\theta_0$, then $g(\vec x)$ is not continuous on any ball $B_{r^\prime}(\vec x_0)$.

  \item If $p/q=\theta_0$, and $\vec x_0$ is an isolated critical point of $P$, and $D(\vec x_0)=0$, then $f(x)$  is not continuous on any ball $B_{r^\prime}(\vec x_0)$.
  \end{enumerate}
\end{thm}
\begin{proof}
  The function $f$ is bounded on $B_r$ by Theorem \ref{thm3.2}.  By construction, $0\le g(\vec x)\le f(\vec x)$ on $\re^n$, with $g(\vec x)=f(\vec x)$ for $\vec x\in(\re^n\setminus\vv(|\vec\nabla P|^2))\cup\vv(D)$.   In the special case $\vv(D)=\vv(|\vec\nabla P|^2)$, $f$ and $g$ are equal everywhere.
  
  Recall $\theta_0$ is the smallest exponent from statement (\L7) of Proposition \ref{prop6.2}, and, from (\L6), let $p_0$, $q_0$ be positive integers so that $\displaystyle{\theta_0=\frac{p_0}{q_0}}$.  With the radius $r_0$ and $C_0$ corresponding to $\theta_0$ in the definition of $\mathbf\Theta$ from (\L5) in Proposition \ref{prop6.2}, for $\vec x\in B_{r_0}$, 
\begin{equation}\label{eq4.5}
  |P(\vec x)|^{\theta_0}\le C_0|\vec\nabla P(\vec x)|.
\end{equation}
The rational function $P^{2p_0}/|\vec\nabla P|^{2q_0}$ is bounded as in line (\ref{eq2}), and Lemma \ref{lem3.4} applies, so that for any $\tau>0$,
$$S(\vec x)=\left\{\begin{array}{ll}{\displaystyle{\frac{(P(\vec x))^{2p_0}}{(|\vec\nabla P(\vec x)|^{2q_0})^{1-\tau}}}} & \mbox{\rm{for} $\vec\nabla P(\vec x)\ne\vec0$}\\ & \\0& \mbox{\rm{for} $\vec\nabla P(\vec x)=\vec0$} \end{array}\right.$$
 is continuous on $B_{r_0}$.

 For statement (1), assuming $p/q=\theta>\theta_0$, let $\displaystyle{\tau=1-\frac{\theta_0}{\theta}>0}$, so that for $\vec\nabla P(\vec x)\ne\vec0$,
 \begin{equation}\label{eq8}
   \frac{(P(\vec x))^{2p_0}}{(|\vec\nabla P(\vec x)|^{2q_0})^{1-\tau}}=\frac{(P(\vec x))^{2p_0}}{|\vec\nabla P(\vec x)|^{2p_0q/p}}=\left(\frac{(P(\vec x))^{2p}}{|\vec\nabla P(\vec x)|^{2q}}\right)^{p_0/p}.
   \end{equation}
This shows $g(x)=(S(\vec x))^{p/p_0}$, so $g$ is continuous on $B_{r_0}$ and we set the claimed $r^\prime=\min\{r,r_0\}$.  We want to show that $f(\vec x)=(S(\vec x))^{p/p_0}$ on $B_{r_0}$.  If $D(\vec x)=0$ then $f(\vec x)=0=(S(\vec x))^{p/p_0}$.  For a point $\vec x\in B_{r_0}$ with $D(\vec x)\ne0$, some of the following limit steps are the same as (\ref{eq4}) and (\ref{eq7}), using a sequence $\vec x_i\in B_{r_0}\setminus\vv(|\vec\nabla P|^{2q})$, the continuity of $N/D$ and $S$ at $\vec x$, and equation (\ref{eq8}).
 \begin{equation}
  f(\vec x)=\frac{N(\vec x)}{D(\vec x)}=\lim_{i\to\infty}\frac{N(\vec x_i)}{D(\vec x_i)}=\lim_{i\to\infty}\frac{(P(\vec x_i))^{2p}}{\left|\vec\nabla P(\vec x_i)\right|^{2q}}=\lim_{i\to\infty}(S(\vec x_i))^{p/p_0}=(S(\vec x))^{p/p_0}.\nonumber
\end{equation}

For statement (2), we continue to use the above notation for $\displaystyle{\theta_0=\frac{p_0}{q_0}}$.  The following argument is the same as Case 2.\ from Theorem \ref{thm3.3}, using (\L7).  Suppose toward a contradiction that $g(\vec x)$ is continuous on some ball $B_{r^{\prime}}$.  Then Lemma \ref{lem3.3} applies: there is some rational $\varphi>0$ so that 
$$g_{\varphi}(\vec x)=\left\{\begin{array}{ll}{\displaystyle{\frac{(P(\vec x))^{2p_0}}{(|\vec\nabla P(\vec x)|^{2q_0})^{1+\varphi}}}} & \mbox{\rm{for} $\vec\nabla P(\vec x)\ne\vec0$}\\ & \\0& \mbox{\rm{for} $\vec\nabla P(\vec x)=\vec0$} \end{array}\right.$$
is continuous, and bounded by $C_\varphi$ on some possibly smaller ball $B_{r^{\prime\prime}}$.
It follows that for $\vec x\in B_{r^{\prime\prime}}\setminus\vv(|\vec\nabla P(\vec x)|^2)$, 
\begin{equation}\label{eq4.6}
|P(\vec x)|^{2p_0/(2q_0(1+\varphi))}\le C_\varphi^{1/(2q_0(1+\varphi))}|\vec\nabla P(\vec x)|.
\end{equation}
If $\vec\nabla P(\vec x)=\vec0$ then both sides of the inequality (\ref{eq4.6}) are $=0$ by the inequality (\ref{eq4.5}), so the inequality (\ref{eq4.6}) holds for all $\vec x\in B_{r^{\prime\prime}}$.  This contradicts the minimality of $\theta_0$.

For statement (3), if there is some $r^{\prime\prime}$ so that $\vv(|\vec\nabla P|^2)\cap B_{r^{\prime\prime}}=\{\vec x_0\}$, then $\vv(|\vec\nabla P|^2)\cap B_{r^{\prime\prime}}=\vv(D)\cap B_{r^{\prime\prime}}$ and so $f=g$ on $B_{r^{\prime\prime}}$.  Given any $B_{r^\prime}$, $g$ is not continuous on $B_{r^\prime}\cap B_{r^{\prime\prime}}$ by (2) and $f$ is not continuous on $B_{r^\prime}$. 
%
%
\end{proof}
\begin{rem}\label{rem7.5}
    The hypothesis in Statement (3) of Theorem \ref{thm3.3} is not too restrictive, considering that for generic $P$, the set of critical points is a finite set of isolated points and contains $\vv(D)$.  However, the extra assumption $D(\vec x_0)=0$ is needed, as shown by Example \ref{ex4.5}.   If $\vv(D)\cap B_{r^{\prime\prime}}=\mbox{\O}$, then $f$ is a continuous, nonvanishing rational function on $B_{r^{\prime\prime}}$.  The case not addressed by Theorem \ref{thm3.3} is whether $f$ is continuous for $p_0/q_0=\theta_0$, $\vec x_0\in\vv(D)\subsetneq\vv(|\vec\nabla P|^2)$, $\dim\vv(|\vec\nabla P|^2)>0$.
\end{rem}

\begin{example}\label{ex4.4}
  The polynomial $P=xy$ satisfies $\vec\nabla P=(y,x)$, $|\vec\nabla P|^2=x^2+y^2$, and at the unique critical point $\vec x_0=\vec0$, $\theta_0=\frac12$.  The rational function $$\frac P{|\vec\nabla P|^2}=\frac{xy}{x^2+y^2}$$ matches Example \ref{ex1.1}; it is bounded and does not have any continuous extension near $\vec0$.  For a rational number $\theta=p/q$, the polynomials $P^{2p}$ and $|\vec\nabla P|^{2q}$ do not have any common factors, $D=(x^2+y^2)^q$ has an isolated zero, and for $\theta\ge\frac12$ (equivalently, $2p\ge q>0$), $$\frac ND=\frac{(xy)^{2p}}{(x^2+y^2)^{q}}$$ is bounded.  $N/D$ extends continuously to $f$ as in Theorem \ref{thm3.2} if and only if $\theta>\frac12$.
\end{example}
\begin{example}\label{ex4.5}
    The polynomial $P=x^2+y^2$ satisfies $\vec\nabla P=(2x,2y)$, $|\vec\nabla P|^2=4(x^2+y^2)$, and $\theta_0=\frac12$ at $\vec0$.  The rational function $$\frac P{|\vec\nabla P|^2}=\frac{x^2+y^2}{4(x^2+y^2)}$$ is a scalar multiple of  Example \ref{ex2.9}.  This is also an instance of Example \ref{ex3.4}.  For integers $p$, $q$ with $2p\ge q>0$, the construction from Example \ref{ex3.2} and Notation \ref{not4.5} gives $$\frac ND=\frac{(x^2+y^2)^{2p-q}}{4^q}$$ so the function $f=\frac ND$ from Theorem \ref{thm3.2} is a polynomial, and in particular continuous, but with limit $0$ only for $2p>q$.  For $2p=q>0$, the function $g$ from Theorem \ref{thm3.2}, $$g(\vec x)=\left\{\begin{array}{ll}{\displaystyle{\frac{(x^2+y^2)^{q}}{4^q(x^2+y^2)^{q}}}} & \mbox{\rm{for} $(x,y)\ne\vec0$}\\ & \\0& \mbox{\rm{for} $(x,y)=\vec0$} \end{array}\right.$$  is not continuous on any ball $B_{r^\prime}$.
\end{example}
\begin{example}\label{ex4.6}
    The polynomial $P=x^5y^2$ satisfies $\vec\nabla P=(5x^4y^2,2x^5y)$, $|\vec\nabla P|^2=25x^8y^4+4x^{10}y^2$.  The \L ojasiewicz exponent at $\vec0$ is $\theta_0=\frac67$.  For integers $p$, $q$ with $7p\ge 6q>0$, the construction from Example \ref{ex3.2} gives bounded rational functions,
    \begin{eqnarray}
        \frac{P^{2p}}{|\vec\nabla P|^{2q}}&=&\frac{(x^{10}y^4)^p}{(x^8y^2(4x^2+25y^2))^q},\label{eq4.1}\\
        \frac ND&=&\frac{x^{10p-8q}y^{4p-2q}}{(4x^2+25y^2)^q}.\label{eq4.2}
    \end{eqnarray}  These numerators and denominators are homogeneous as in Example \ref{ex3.4}, and this $P$ has the exceptional property that the set of critical points is a $1$-dimensional set; the indeterminacy set of (\ref{eq4.1}) is $\{xy=0\}$.  After the cancellation to get the expression (\ref{eq4.2}), $\vv(D)$ contains only the isolated point at the origin, as in Theorem \ref{thm3.2}.  For $p/q=6/7$, (\ref{eq4.1}) and (\ref{eq4.2}) are locally bounded but the extensions $f$ and $g$ are equal to each other and are not continuous.
\end{example}


\begin{example}\label{ex4.7}
    The polynomial $P=x^2-y^3$ is irreducible and satisfies $\vec\nabla P=(2x,-3y^2)$, and $|\vec\nabla P|^2=4x^2+9y^4$ has an isolated zero, as in Example \ref{ex1.2} and Example \ref{ex3.5}.  The \L ojasiewicz exponent at $\vec0$ is $\theta_0=\frac23$.  For integers $p$, $q$ with $3p=2q>0$, the construction from Example \ref{ex3.2} gives locally bounded rational functions,
    \begin{equation*}
        \frac{P^{2p}}{|\vec\nabla P|^{2q}}=\frac ND=\frac{(x^2-y^3)^{2p}}{(4x^2+9y^4)^q}.
    \end{equation*} 
    It is easy to check (by finding limits along the axes or plotting a graph) that the extension $f=g$ of the rational function $$\displaystyle{\frac{(x^2-y^3)^{4}}{(4x^2+9y^4)^3}}$$ is not continuous in any neighborhood of the origin.  The graph will also show the property of local boundedness, but we can check with some elementary algebra.  For $x=0$, $y\ne0$, $g(0,y)=9^{-3}$.  For $x\ne0$, $y=0$, $g(x,0)=x^2/4^3$.  For $x\ne0$ and $y\ne0$,
    \begin{eqnarray*}
        (g(x,y))^{1/4}&=&\frac{|x^2-y^3|}{(4x^2+9y^4)^{3/4}}\le\frac{x^2}{4^{3/4}|x|^{3/2}}+\frac{|y|^3}{9^{3/4}|y|^3}.
    \end{eqnarray*}
\end{example}

\begin{thm}\label{ex4.11}
  For any real algebraic set $A$ strictly contained in a hyperplane, and any $\vec x_0\in A$, there exist $r>0$ and a rational function $N/D$, so that the indeterminacy set of $N/D$ is equal to $A$, and as in Theorem {\rm{\ref{thm3.3}}}, $N$ and $D$ have no common factors, $N/D$ is bounded on $B_r(\vec x_0)$, and the extension $f(\vec x)$ of $N/D$ is not continuous on any neighborhood of $\vec x_0$.
\end{thm}
\begin{proof}
  For $n\ge2$, we choose coordinates so that the hyperplane is $\{(x_1,\ldots,x_{n-1},0)\}$.  The real algebraic set $A$ of the hypothesis is of the form $\vv(P)$, for $P(\vec x)=Q^2+x_n^2$ and $Q=Q(x_1,\ldots,x_{n-1})\not\equiv0$.  A short calculation of the squared norm of the gradient gives:
  \begin{equation}\label{eq4.11}
    |\vec\nabla P|^2=4Q^2|\vec\nabla Q|^2+4x_n^2.
  \end{equation}
  We note that both $P$ and $|\vec\nabla P|^2$ are irreducible polynomials in the polynomial ring $\re[\vec x]$, so they have no non-constant polynomial common factor unless one is a polynomial multiple of the other.  This can only happen when one is a scalar multiple of the other and $Q$ is a linear polynomial with $|\vec\nabla Q|^2\equiv1$ (as in Example \ref{ex4.5}), and this case can be avoided by replacing $Q$ with $2Q$, which does not change the set $A$.

  If $P(\vec x)=0$ then $Q(x_1,\ldots,x_{n-1})=0$ and $x_n=0$, so by Equation (\ref{eq4.11}), $|\vec\nabla P|^2=0$, and $\vv(P)\subseteq\vv(|\vec\nabla P|^2)$. 
  
    For any $p$ and $q$ so that $p/q=\theta_0$ is the \L ojasiewicz exponent for $P$ at $\vec x_0$, there is some $r>0$ so that the rational function \begin{equation}\label{eq4.14}
    \displaystyle{\frac{(P(\vec x))^{2p}}{|\vec\nabla P(\vec x)|^{2q}}}=\frac{(Q^2+x_n^2)^{2p}}{(4Q^2|\vec\nabla Q|^2+4x_n^2)^q}=\frac{N(\vec x)}{D(\vec x)}
    \end{equation}
    is bounded on $B_{r}(\vec x_0)$ and has indeterminacy set $A$ (with dimension $\le n-2$ and $\vv(D)\cap B_r=A\cap B_r$ as in Theorem \ref{thm3.2}).  By construction, $(P(\vec x))^{2p}$ and $|\vec\nabla P(\vec x)|^{2q}$ have no common factors.  The extension $f(\vec x)=g(\vec x)$ from Theorem \ref{thm3.3} is not continuous on any neighborhood of $\vec x_0$.  We emphasize that the number $\theta_0$, and, consequently, the radius $r$, the expression (\ref{eq4.14}), and the extension $f$, all depend on the point $\vec x_0$ and the choice of defining function $Q$, but the indeterminacy locus of the expression (\ref{eq4.14}) in $\re^n$ does not.
\end{proof}

\begin{example}\label{ex5.12}
  As in Example \ref{ex4.7}, let $Q(x,y)=x^2-y^3$, so the construction from Theorem \ref{ex4.11} gives irreducible polynomials:
  \begin{eqnarray*}
      P(x,y,z)&=&(x^2-y^3)^2+z^2\\
      |\vec\nabla P|^2&=&4(x^2-y^3)^2(4x^2+9y^4)+4z^2.
  \end{eqnarray*}
  The set $\vv(P)$ has codimension $2$ in $\re^3$; it is a real curve in the $xy$-plane with a cusp at the origin, and exactly equal to the set of critical points of $P$.  The \L ojasiewicz exponent of $P$ at $(0,0,0)$ is $\theta_0=5/6$, and by Theorem \ref{thm3.3}, the function $g:\re^3\to\re$
$$g(\vec x)=\left\{\begin{array}{ll}{\displaystyle{\frac{((x^2-y^3)^2+z^2)^{10}}{(4(x^2-y^3)^2(4x^2+9y^4)+4z^2)^{6}}}} & \mbox{\rm{for} $(x,y,z)\notin\vv(P)$}\\ & \\0& \mbox{\rm{for} $(x,y,z)\in\vv(P)$} \end{array}\right.$$
  is bounded on some $B_r(\vec0)$ and not continuous on any ball $B_{r^\prime}(\vec0)$.
  These limits at the origin are unequal:
  \begin{eqnarray*}
      \lim_{y\to0}g(0,y,0)&=&\lim_{y\to0}\frac{y^{60}}{((4y^6)(9y^4))^6}=\frac1{36^6},\\
      \lim_{z\to0}g(0,0,z)&=&\lim_{z\to0}\frac{z^{20}}{(4z^2)^6}=0.
  \end{eqnarray*}
  To check that $g$ is bounded in a neighborhood of $\vec0$, consider the following cases.  By construction, $g(\vec x)=0$ on $\vv(P)$.  For $(x,y,z)$ with $x^2-y^3=0$ and $z\ne0$, $g(\vec x)=z^{20}/(4^6z^{12})$. 
  For $x^2-y^3\ne0$ and $z=0$, $$g(x,y,0)=\frac{(x^2-y^3)^{8}}{4^6(4x^2+9y^4)^6},$$ which is bounded near the origin as in Example \ref{ex4.7}.  For all other $\vec x\notin\vv(P)$,
  \begin{eqnarray}
      (g(\vec x))^{1/10}&=&{\displaystyle{\frac{(x^2-y^3)^2+z^2}{(4(x^2-y^3)^2(4x^2+9y^4)+4z^2)^{6/10}}}}\nonumber\\
      &\le&\frac{(x^2-y^3)^2}{(4(x^2-y^3)^2(4x^2+9y^4))^{6/10}}+\frac{z^2}{(4z^2)^{6/10}}\nonumber\\
      &=&\frac{(x^2-y^3)^{4/5}}{4^{3/5}(4x^2+9y^4)^{3/5}}+\frac{z^{4/5}}{4^{3/5}}.\label{eq5.11}
  \end{eqnarray}
  The boundedness of the first term in line (\ref{eq5.11}) is again as in Example \ref{ex4.7}.
  
  In this example, $g$ is continuous in a neighborhood of any point $\vec x_1=(x_1,y_1,0)$ in the indeterminacy set $\vv(P)$, except the origin.  This can be explained by saying the \L ojasiewicz exponent for $P$ at $\vec x_1\ne\vec0$ is less than $\frac56$, but we can check directly.  First, $$\lim_{\vec x\to\vec x_1}4x^2+9y^4=4x_1^2+9y_1^4=L>0,$$ so there is some $\delta>0$ so that if $|\vec x-\vec x_1|<\delta$ then $4x^2+9y^4>L/2$.  Let $\Lambda=\min\{2L,4\}$.  For $\vec x$ within $\delta$ of $\vec x_1$, $g(\vec x)$ is either $0$, or $\vec x\notin\vv(P)$ and 
  \begin{eqnarray*}
      0\le g(x)&\le&{\displaystyle{\frac{((x^2-y^3)^2+z^2)^{10}}{(4(x^2-y^3)^2(L/2)+4z^2)^{6}}}}\\
      &\le&{\displaystyle{\frac{((x^2-y^3)^2+z^2)^{10}}{(\Lambda(x^2-y^3)^2+\Lambda z^2)^{6}}}}={\displaystyle{\frac{((x^2-y^3)^2+z^2)^{4}}{\Lambda^{6}}}},
  \end{eqnarray*}
  which approaches $0$ as $\vec x\to\vec x_1$.
\end{example}

\begin{proof}[Acknowledgment] 
  The graphics in Figure 1 were produced using the free plotting application at {\tt desmos.com}.
\end{proof}

%

\end{document}